\documentclass[a4paper,11pt]{article}
\usepackage[utf8x]{inputenc}

\usepackage{a4wide}
\usepackage{amssymb}
\usepackage{amsmath}
\usepackage{amsthm}
\usepackage[mathscr]{euscript}
\usepackage[normalem]{ulem}
\usepackage[colorlinks=true]{hyperref}
\usepackage[usenames,dvipsnames]{xcolor}
\usepackage{appendix}
\usepackage{tikz}
\usetikzlibrary{arrows,angles}
\usetikzlibrary{matrix}
\usetikzlibrary{decorations}
\usetikzlibrary{arrows,calc,shapes,decorations.pathreplacing}
\usepackage{tikz-cd}
\usepackage{todonotes}

\numberwithin{equation}{section}
\theoremstyle{definition}

\newtheorem{Definition}{Definition}[section]

\newtheorem{Example}[Definition]{Example}
\newtheorem{Conjecture}[Definition]{Conjecture}

\theoremstyle{plain}

\newtheorem{Proposition}[Definition]{Proposition}
\newtheorem{Lemma}[Definition]{Lemma}
\newtheorem{Corollary}[Definition]{Corollary}

\newcommand{\R}{\mathbb R}
\newcommand{\N}{\mathbb N}

\usepackage[xcolor]{changebar}
\cbcolor{blue}
\DeclareMathOperator{\Ric}{Ric}
\usepackage[inline]{enumitem}
\newcommand{\enumlabelformat}{\roman}
\newcommand{\enumlabelfont}[1]{#1}
\newlength{\thelabelsep}
\setlist{labelsep=\thelabelsep}
\setlist[enumerate,1]{font=\enumlabelfont,label=(\enumlabelformat*),leftmargin=2.5em}
\setlist[itemize]{leftmargin=2.5em,label=$-$}

\newcounter{inlineenum}
\renewcommand{\theinlineenum}{\enumlabelformat{inlineenum}}
\let\epsilon\varepsilon
\let\phi\varphi

\newcommand{\argam}[1]{{\color{green!70!black}{#1}}}

\makeatletter
\let\save@mathaccent\mathaccent
\newcommand*\if@single[3]{%
  \setbox0\hbox{${\mathaccent"0362{#1}}^H$}%
  \setbox2\hbox{${\mathaccent"0362{\kern0pt#1}}^H$}%
  \ifdim\ht0=\ht2 #3\else #2\fi
  }
\newcommand*\rel@kern[1]{\kern#1\dimexpr\macc@kerna}
\newcommand*\widebar[1]{\@ifnextchar^{{\wide@bar{#1}{0}}}{\wide@bar{#1}{1}}}
\newcommand*\wide@bar[2]{\if@single{#1}{\wide@bar@{#1}{#2}{1}}{\wide@bar@{#1}{#2}{2}}}
\newcommand*\wide@bar@[3]{%
  \begingroup
  \def\mathaccent##1##2{%
    \let\mathaccent\save@mathaccent
    \if#32 \let\macc@nucleus\first@char \fi
    \setbox\z@\hbox{$\macc@style{\macc@nucleus}_{}$}%
    \setbox\tw@\hbox{$\macc@style{\macc@nucleus}{}_{}$}%
    \dimen@\wd\tw@
    \advance\dimen@-\wd\z@
    \divide\dimen@ 3
    \@tempdima\wd\tw@
    \advance\@tempdima-\scriptspace
    \divide\@tempdima 10
    \advance\dimen@-\@tempdima
    \ifdim\dimen@>\z@ \dimen@0pt\fi
    \rel@kern{0.6}\kern-\dimen@
    \if#31
      \overline{\rel@kern{-0.6}\kern\dimen@\macc@nucleus\rel@kern{0.4}\kern\dimen@}%
      \advance\dimen@0.4\dimexpr\macc@kerna
      \let\final@kern#2%
      \ifdim\dimen@<\z@ \let\final@kern1\fi
      \if\final@kern1 \kern-\dimen@\fi
    \else
      \overline{\rel@kern{-0.6}\kern\dimen@#1}%
    \fi
  }%
  \macc@depth\@ne
  \let\math@bgroup\@empty \let\math@egroup\macc@set@skewchar
  \mathsurround\z@ \frozen@everymath{\mathgroup\macc@group\relax}%
  \macc@set@skewchar\relax
  \let\mathaccentV\macc@nested@a
  \if#31
    \macc@nested@a\relax111{#1}%
  \else
    \def\gobble@till@marker##1\endmarker{}%
    \futurelet\first@char\gobble@till@marker#1\endmarker
    \ifcat\noexpand\first@char A\else
      \def\first@char{}%
    \fi
    \macc@nested@a\relax111{\first@char}%
  \fi
  \endgroup
}
\makeatother
\title{Failure of local equi-Lipschitzness for families of Lorentz distances to Cauchy surface foliations}

\author{Gregory J. Galloway\footnote{Department of Mathematics, University of Miami, Coral Gables, FL, USA.\\galloway@math.miami.edu}\\Robert J. McCann\footnote{Departments of Mathematics and of Economics, University of Toronto, Toronto, ON, Canada.\\ mccann@math.toronto.edu}\\Argam Ohanyan\footnote{Department of Mathematics, University of Toronto, Toronto, ON, Canada.\\argam.ohanyan@utoronto.ca}}
\begin{document}

\date{\today}

%\date{Received: date /Accepted: date}

\maketitle

\begin{abstract}
The families of Lorentzian distance functions to and from points along a complete timelike line in a spacetime are known to be locally equi-Lipschitz continuous in a neighborhood of the line. This is an essential component in the proof of the classical Lorentzian splitting theorems. We show that this property fails in general for families of Lorentzian distances to and from the level sets of a given Cauchy temporal function. Moreover, we formulate conjectures based on the existence of Cauchy temporal functions in cosmological and timelike geodesically complete spacetimes such that the Lorentz distances to its level sets are equi-Lipschitz, which are equivalent to Bartnik's splitting conjecture.
\vspace{1em}

\medskip


\end{abstract}
%\newpage
\tableofcontents
%\newpage

\section{Introduction}
\label{section: introduction}

The Lorentzian splitting theorems \cite{Eschenburg1988, Galloway1989, Newman1990} are by now a classical topic in Lorentzian geometry and mathematical general relativity. They assert that a timelike geodesically complete or globally hyperbolic spacetime $(M,g)$ which contains a complete timelike line (i.e., a proper time-maximizing geodesic $\gamma: \R \to M$) and satisfies the strong energy condition (i.e., $\Ric(v,v) \geq 0$ for all timelike $v \in TM$) splits isometrically as a product spacetime $(\R \times S, dt^2 - \tilde g)$, where $(S,\tilde g)$ is a complete Riemannian manifold of nonnegative Ricci curvature. Recently, new methods were developed in \cite{Quintetsmooth} based on the ellipticity of the $p$-d'Alembertian studied in \cite{octet}, which lead to a vastly improved Lorentz--Finsler splitting theorem \cite{caponio2024splitting} as well as a splitting for spacetimes whose metric and volume weight can be of $C^1$ regularity \cite{quintetnonsmooth}.

A central ingredient in the proof is the following: Given a complete timelike line $\gamma: \R \to M$ in a timelike geodesically complete or globally hyperbolic spacetime $(M,g)$, there exists a neighborhood $U$ of $\gamma(\R)$ such that the functions
\begin{equation}
    \{\ell(\cdot, \gamma(t))\}_{t > 0}, \quad \{\ell(\gamma(-t),\cdot)\}_{t > 0},
\end{equation}
are {locally} equi-Lipschitz continuous on $U$.

While this result already appears in Eschenburg \cite{Eschenburg1988}, it is best understood in terms of the \emph{(generalized) timelike co-ray condition} introduced in Galloway and Horta \cite{GallowayHorta}, which, if it holds at a point $q$, states that curves $\alpha:[0,\infty) \to M$ obtained as subsequential limits of (almost) maximizers connecting $q$ to $\gamma(t_n)$ (or $\gamma(-t_n)$ to $q$), are timelike. One can show that this is an open condition and on $\gamma$ itself all such limits only produce parts of the timelike line $\gamma$. Thus the condition holds in a neighborhood of $\gamma$.

In situations where one would like to prove a splitting of spacetime, but there is no complete timelike line present a priori, such as in the setting of Bartnik's \cite{Bartnik1988} cosmological splitting conjecture, one might be tempted to consider families of Lorentz distances $\{\ell(\cdot,\Sigma_t)\}_{t > 0}$ as well as $\{\ell(\Sigma_{-t},\cdot)\}_{t > 0}$, where $\Sigma_s :=\{\tau = s\}$ are the level sets of a given Cauchy temporal function $\tau \in C^\infty(M)$ and are thus smooth spacelike Cauchy hypersurfaces. In this note, we show that there is in general an issue with such an approach, even for rather nice (i.e., steep and $h$-steep with respect to a background complete Riemann metric $h$) Cauchy temporal functions, by way of two counterexamples.\footnote{Some partial efforts to deal with these issues for a specific class of Cauchy surfaces have been made in \cite{GVAchronal, GVHausdorff}, whereby specific splitting results are obtained, subject to certain additional conditions.} These examples also contradict Lemma 3.2 in \cite{ZhuWuCui} and subsequently the results therein which assert the existence of global viscosity solutions to the Lorentzian eikonal equation $g(\nabla u,\nabla u) = 1$. 

It is worth noting that the existence of a Cauchy temporal function for which the aforementioned families are locally equi-Lipschitz is equivalent to Bartnik's splitting conjecture.

\subsection{Notation and conventions}
\label{subsection: notation conventions}

Lorentzian metrics have signature $(+,-,\dots,-)$, so
$g(v,v) > 0$ means $v$ is timelike, $g(v,v) = 0$ and $v \neq 0$ means $v$ is null, and $g(v,v) < 0$ or $v = 0$ means $v$ is spacelike.
Spacetimes are connected, time oriented Lorentzian manifolds 
with smooth metric tensors $g_{ij} \in C^\infty$ 
and have dimension $\dim M =:n \geq 2$. 
We fix a background complete Riemannian metric $h$ throughout. If $X$ is the time orientation vector field and $v$ is a causal vector, then $g(v,X) > 0$ means $v$ is future, $g(v,X) < 0$ means $v$ is past. If $v$ is causal, we write $|v|_g:=\sqrt{g(v,v)}$ for its Lorentzian norm. A spacelike hypersurface is a hypersurface such that the restriction of $g$ on it is negative definite. The function $\ell:M^2 \to \{-\infty\} \cup [0,\infty)$ denotes the (extended) time separation function, by convention $\ell(x,y) = -\infty$ if $x \not \leq y$. Moreover, we write $\ell(A,B):=\sup_{x \in A, y\in B} \ell(x,y)$. A function $\tau \in C^\infty(M)$ is called \emph{temporal} if $\nabla \tau$ is future directed timelike. A temporal function $\tau$ is \emph{Cauchy} if each of its level sets $\{\tau = s\}$, $s \in \R$, is a Cauchy hypersurface.

\section{Busemann functions associated to timelike lines}
\label{Section: Proof of Bartnik}

Let $(M,g)$ be globally hyperbolic or timelike geodesically complete, and let $\gamma:\R \to M$ be a complete timelike line. Let us briefly describe how to obtain the local equi-Lipschitz regularity of the family $\{\ell(\cdot,\gamma(t))\}_{t > 0}$ in a neighborhood of $\gamma(\R)$, following Galloway--Horta \cite{GallowayHorta}.

Given a point $q \in I(\gamma):=I^+(\gamma) \cap I^-(\gamma)$, a \emph{(generalized) causal co-ray} is an inextendible causal curve $\alpha$ starting at $ q$ which is obtained as a subsequential limit of (almost) maximizing timelike curves $\alpha_n$ which connect $q$ to $\gamma(t_n)$, where $t_n \to +\infty$. Such a (generalized) causal co-ray is always causal and maximizing. We say the \emph{(generalized) timelike co-ray condition} holds at $q$ if every (generalized) causal co-ray starting at $q$ is timelike. The essential result now is the following

\begin{Proposition}[\cite{GallowayHorta}, Prop.\ 5.1; Cor.\ 5.2]
    \begin{enumerate}
        \item[]
        \item The (generalized) timelike co-ray condition is an open condition.
        \item Every (generalized) causal co-ray starting at $\gamma(t)$, for any $t$, corresponds to $\gamma|_{[t,\infty)}$. In particular, the (generalized) timelike co-ray condition holds at every point on $\gamma$.
    \end{enumerate}
\end{Proposition}

The above result implies that the (generalized) timelike co-ray condition holds in a neighborhood $U$ of $\gamma(\R)$. In particular, it is now easily seen that $\{\ell(\cdot,\gamma(t)\}_{t > 0}$ is equi-Lipschitz on this neighborhood. If there existed $q \in U$ such that this family was not equi-Lipschitz in any neighborhood of $q$, one could find (almost) maximizing timelike geodesics in $g$-arclength parametrization from $q_n $ to $\gamma(t_n)$, where $q_n \to q$ and $t_n \to \infty$, such that $|\gamma_n'(0)|_h \to \infty$. In this case, it is easily seen that the $\gamma_n$ must converge subsequentially to a null ray at $q$, thus contradicting the (generalized) timelike co-ray condition.

In fact, the family $\{\ell(\cdot,\gamma(t)\}_{t > 0}$ is not only locally equi-Lipschitz but also locally equi-semiconvex on $U$ \cite[Prop.\ 5]{Quintetsmooth}; (or equi-superdifferentiable \cite[Prop.\ 13]{quintetnonsmooth} in less smooth settings). These ingredients are essential in the use of $p$-d'Alembertian techniques, where we recall that $\square_p f:=-\mathrm{div}( |\nabla f|^{p-2}\nabla f)$, $0 \neq p < 1$. In \cite{Quintetsmooth}, we give a streamlined argument (the result was also shown in \cite{octet}) for the fact that the standard smooth d'Alembertian comparison (which is a consequence of the strong energy condition, i.e., $\Ric(v,v) \geq 0$ for all timelike vectors $v$) for a Lorentz distance $f:=\ell(\cdot,y)$ on $I^-(y) \setminus \mathrm{Cut}^-(y)$ extends to a weak comparison inequality across the cut locus, in the sense that
\begin{equation}
    \int_M \left( \frac{(n-1)\phi}{f} + g(\nabla \phi, |\nabla f|^{p-2} \nabla f) \right) \, \mathrm{dvol}_g\geq 0
\end{equation}
for all $0 \leq \phi \in C^1_c(I^-(y))$. Applying this result now to the sequence of approximate future Busemann functions
\begin{equation}
    b_t^+(x):=\ell(\gamma(0),\gamma(t)) - \ell(x,\gamma(t))
\end{equation}
and keeping in mind the local equi-Lipschitzness and local equi-semiconvexity of the family $\{\ell(\cdot,\gamma(t))\}_{t > 0}$ in the neighborhood of the line where the (generalized) timelike co-ray condition holds, we may pass to the limit to obtain
\begin{equation}
    \int_M g(\nabla \phi, |\nabla b^+|_g^{p-2} \nabla b^+) \,  \mathrm{dvol}_g \leq 0,
\end{equation}
where $b^+:=\lim_{t \to \infty} b_t^+$ is the \emph{future Busemann function} associated with $\gamma$. This can be understood as the weak form of the inequality $\square_p b^+ \leq 0$. Similarly, $\square_p b^- \geq 0$, where $b^-:=\lim_{t \to \infty} b_t^-$ is the \emph{past Busemann function} associated with $\gamma$ and
\begin{equation}
    b_t^-(x):=\ell(\gamma(-t),x) - \ell(\gamma(-t),\gamma(0)).
\end{equation}
Since $b^+ \geq b^-$ and $b^+(\gamma(t)) = b^-(\gamma(t))$ for every $t$ due to the fact that $\gamma$ is a maximizing timelike line, one can obtain $b^+ = b^-$ on a neighborhood of $\gamma(\R)$ by the maximum principle for quasilinear elliptic operators, cf.\ \cite[Prop.\ 9]{Quintetsmooth}. An application of a nonlinear Bochner-type identity in conjunction with the strong energy condition then yields $\mathrm{Hess}_g(b^+) = 0$ and in particular that $b^+$ is smooth near $\gamma$, see \cite[Cor.\ 14]{Quintetsmooth}. From here, the remaining arguments to obtain the Lorentzian splitting theorem follow the classical sources \cite{Eschenburg1988, Galloway1989, Newman1990, GallowayHorta}.

One may be tempted to follow this line of thought using Busemann-type functions associated to Cauchy surface foliations, but, as  the next section shows, the equi-Lipschitz property of Lorentz distances fails in this case in general, thus the analogous arguments do not go through due to lack of regularity.

\section{Busemann type functions associated to a Cauchy surface foliation}

Let $(M,g)$ be a globally hyperbolic spacetime. Let $\tau \in C^\infty(M)$ be a given steep and $h$-steep temporal function, i.e., 
\begin{equation}
\label{eq: steep and h steep}
    d\tau(V) \geq \max(|V|_h, |V|_g).
\end{equation}
Such a function $\tau$ is automatically a Cauchy temporal function. Bernard--Suhr \cite{BernardSuhr} show that any globally hyperbolic spacetime possesses such a steep and $h$-steep Cauchy temporal function. Now, \cite[Lemma 3.2]{ZhuWuCui} (and subsequent results) assert the local equi-Lipschitzness of the families of Lorentz distances $\{\ell(\cdot,\Sigma_t)\}_{t > 0}$ as well as $\{\ell(\Sigma_{-t},\cdot)\}_{t > 0}$, where $\Sigma_s:=\{\tau = s\}$. The aforementioned reference uses this in an essential way in their construction of global viscosity solutions to the Lorentzian eikonal equation $g(\nabla u,\nabla u) = 1$. However, we show that it does not hold in that generality without further assumptions. In fact, one can even find a temporal function which is steep with respect to the background Euclidean metric on $2$-dimensional Minkowski spacetime which contradicts the claim. Before coming to it, we note the following elementary observation:

\begin{Lemma}
\label{Lemma: blowup of initial tangents}
The family $\{\ell(\cdot,\Sigma_t)\}_{t > 0}$ is locally equi-Lipschitz on $M$ if and only if $M$ is covered by open sets $U$ such that for each $q \in U$, the $h$-norms of the initial tangents of $g$-unit speed maximizing timelike geodesics $\gamma_{(t)}$ from $q$ to $\Sigma_t$ satisfy $|\gamma'_{(t)}(0)|_h \leq C$ for some $C > 0$ which depends on $U$ but is independent of $t$, and $t$ is large enough so that $\Sigma_t \subseteq I^+(q)$ for every $q \in U$.
\end{Lemma}

\begin{Example}
We give a counterexample to 
\cite[Lemma 3.2]{ZhuWuCui}. Consider Minkowski space $\R^{1,1}$ with metric $dt^2 - dx^2$. If we transform to null coordinates $(u,v)$ which are defined via
\begin{equation}
    u = \frac{t - x}{\sqrt{2}}, \quad v = \frac{t + x}{\sqrt{2}},
\end{equation}
the Minkowski metric is $g = 2dudv$ in these coordinates.
A vector $X = A \partial_u + B \partial_v$ is future causal if and only if $A,B \geq 0$ and $\max(A,B) > 0$, since $g(X,X) = 2AB$. We write $h = du^2 + dv^2$ for the standard Euclidean metric in these coordinates. Define now
\begin{equation}
    \tau(u,v):=F(u) + v,
\end{equation}
where $F$ is a smooth function satisfying $F'(u) \geq 1$ for all $u$ (we will specify a choice of $F$ further below). We claim that $\tau$ is a steep and $h$-steep temporal function: Indeed, $d\tau = F'(u) du + dv$, thus $\nabla \tau = {\partial_u} + F'(u) \partial_v$, so that $g(\nabla \tau, \nabla \tau) = 2 F'(u) \geq 2$, so $\tau$ is steep temporal. To see that it is $h$-steep, take any future causal $X$ as above, then
\begin{equation}
\label{eq: tau hsteep}
    d\tau (X) = F'(u) du(X) + dv(X) = F'(u) A + B \geq A + B \geq \sqrt{A^2 + B^2} = |X|_h.
\end{equation}
All in all, $\tau$ is a steep and $h$-steep (thus Cauchy) temporal function, cf.\ \eqref{eq: steep and h steep}. Consider now the level set $\Sigma_s = \{(u,v) : F(u) + v = s\}$. Let us calculate the intersection of $J^+(0)$ with $\Sigma_s$. Take any straight line $\gamma(t) = (tu,tv)$ with $(u,v) \in \Sigma_s$, for $\gamma$ to be future causal we must have $u,v \geq 0$, but $v = s - F(u)$, so that $F(u) \leq s$, as well as $u \geq 0$. By monotonicity, there is a unique $u_{max}$ such that $F(u_{max}) = s$, and $F(u) \leq s$ for all $u \in [0,u_{max}]$. All in all,
\begin{equation}
    J^+(0) \cap \Sigma_s = \{(u,v) : u \in [0,u_{max}], 0 \leq v = s - F(u)\},
\end{equation}
which is compact. Take now a maximizing causal line from $\gamma(t) = (tu,tv) =(tu,t(s - F(u)))$ connecting the origin to some point on $\Sigma_s$, we want to determine at which point on $\Sigma_s$ the time separation to $\Sigma_s$ from the origin is maximized. To this end, note that
\begin{equation}
    \frac12\ell((0,0),(u,v))^2 = uv = u(s - F(u)) =:f(u).
\end{equation}
Let us now make a specific choice $F(u) = u + u^3/3$, which indeed satisfies $F'(u) = 1 + u^2 \geq 1$ as required. Moreover, $F''(u) = 2u$. We claim that with this choice of $F$, $f(u)$ is strictly concave: indeed
\begin{align}
    f'(u) &= s - F(u) - u F'(u),\\
    f''(u) &= - 2 F'(u) -  u F''(u) = -  (2 + 2u^2 + 2u^2) < 0.
\end{align}
Let $s$ be large enough so that $s > F(0)$. Also, note that $f'(u_{max}) = s - F(u_{max}) -  u_{max} F'(u_{max}) = - u_{max} F'(u_{max}) < 0.$ Thus, $f'(u) = 0$ uniquely at some $u_s \in (0,u_{max})$, where $u_s$ is implicitly given by
\begin{equation}
    s = F(u_s) + u_sF'(u_s)
\end{equation}
and the pair $(u_s, s - F(u_s) = v_s)$ is the point on $\Sigma_s$ at which a maximizer from the origin to $\Sigma_s$ lands. The maximizing segment is precisely $ t \in [0,1] \mapsto \gamma(t) = (tu_s,tv_s)$. Note that, using the explicit form of $F$ and the fact that $f'(u_s) = 0$, we get
\begin{equation}
    v_s = s - F(u_s) = u_s F'(u_s) = u_s (1+ u_s^2).
\end{equation}
Thus,
\begin{equation}
    \gamma'(0) = u_s \partial_u + u_s (1 + u_s^2) \partial_v.
\end{equation}
Hence
\begin{equation}
    g(\gamma'(0),\gamma'(0)) = 2u_s^2 (1 + u_s^2).
\end{equation}
If we write $W_s:= \gamma'(0)/|\gamma'(0)|_g$, then
\begin{equation}
    W_s = \frac{1}{\sqrt{2(1+u_s^2)}} \partial_u + \sqrt{\frac{1 + u_s^2}{2}} \partial_v.
\end{equation}
Then
\begin{equation}
    h(W_s,W_s) = \frac{1}{2 (1 + u_s^2)} + \frac{1 + u_s^2}{2}.
\end{equation}
Clearly, the above goes to $+\infty$ as $u_s \to +\infty$, or equivalently, as  $s \to +\infty$ since
\begin{equation}
    s = F(u_s) + u_s F'(u_s) =  u_s + u_s^3/3 + u_s(1 + u_s^2) = 2 u_s +  4 u_s^3/3.
\end{equation}
\end{Example}

The result continues to be false even if one assumes the Cauchy surfaces to be compact. Before coming to this, we observe that the equi-Lipschitzness already forces the existence of timelike lines in this setting.

\begin{Lemma}
\label{L:equi-Lipschitz implies splitting}
    Let $(M,g)$ be a globally hyperbolic spacetime with compact Cauchy surfaces. Let $\tau \in C^\infty(M)$ be a Cauchy temporal function with level sets $\Sigma_s:=\{\tau = s\}$. If either of the families $\{\ell(\cdot,\Sigma_t)\}_{t > 0}$ or $\{\ell(\Sigma_{-t},\cdot)\}_{t > 0}$ is locally equi-Lipschitz, then there exists a (not necessarily complete) timelike line in $M$.
\end{Lemma}

\begin{proof}
    Let us assume $\{\ell(\cdot,\Sigma_t)\}_{t > 0}$ is locally equi-Lipschitz, the other case is analogous. 
    
    Compactness yields a $g$-unit speed maximizing timelike geodesic $\gamma_n$ connecting $\Sigma_{-n}$ to $\Sigma_n$. Moreover, $\gamma_n$ intersects $\Sigma_0$ uniquely in a point $z_n$ (normalized so that $z_n = \gamma_n(0)$). By compactness, upon passing to a subsequence, we may assume that $z_n \to z \in \Sigma_0$. By assumption, since $\gamma_n$ also maximizes from $z_n$ to $\Sigma_n$, the local equi-Lipschitzness of $\{\ell(\cdot,\Sigma_t)\}_{t > 0}$ implies (cf.\ Lemma \ref{Lemma: blowup of initial tangents}) that the sequence $\{\gamma'_n(0)\}_{n \in \N}$ is contained in a compact subset of the future unit timelike bundle of $M$, thus, after passing to a further subsequence, $\gamma_n'(0) \to v \in T_{z}M$ and $g(v,v) = 1$. By standard ODE theory, $\gamma_n \to \gamma$ in $C^\infty_{loc}$, where $\gamma$ is the maximally extended geodesic with initial data $\gamma(0) = z$ and $\gamma'(0) = v$. As all of the $\gamma_n$ are maximizing, $\gamma$ is thus a timelike line.
\end{proof}

\begin{Example}
    In \cite{EhrlichGallowayTimelikelines}, Ehrlich--Galloway construct a timelike (and null) geodesically complete, globally hyperbolic spacetime with compact Cauchy surfaces which does not satisfy the strong energy condition and which contains no timelike lines. Thus, by the previous result, there is no Cauchy temporal function $\tau$ on this spacetime (no matter if ($h$-)steep or not) for which either of the families $\{\ell(\cdot,\Sigma_t)\}_{t > 0}$ or $\{\ell(\Sigma_{-t},\cdot)\}_{t > 0}$ is locally equi-Lipschitz.
\end{Example}

Let us note the following connection of the equi-Lipschitz property of the families of Lorentz distances with Bartnik's cosmological splitting conjecture:

\begin{Corollary}
\label{P:old}
    Let $(M,g)$ be a globally hyperbolic, timelike geodesically complete spacetime with compact Cauchy surfaces satisfying the strong energy condition. Then $(M,g)$ splits isometrically as a Lorentzian product if and only if there exists a Cauchy temporal function $\tau \in C^\infty(M)$ such that either of the families $\{\ell(\cdot,\Sigma_t)\}_{t > 0}$ or $\{\ell(\Sigma_{-t},\cdot)\}_{t > 0}$ is locally equi-Lipschitz.
\end{Corollary}
\begin{proof}
The `if' statement follows from Lemma \ref{L:equi-Lipschitz implies splitting}. 
The `only if' statement is elementary to verify since in this case the geometrical splitting is assumed;  (for a stronger conclusion under a stronger hypothesis,  see Proposition \ref{P:compact splitting implies equi-Lipschitz} below).
\end{proof}

\section{Outlook}
\label{section: outlook}

In light of the relationship between equi-Lipschitzness of Lorentz distances to Cauchy surfaces and Bartnik's splitting conjecture, let us formulate the following conjecture which is equivalent to Bartnik's:

\begin{Conjecture}
    Let $(M,g)$ be a globally hyperbolic, timelike geodesically complete spacetime with compact Cauchy surfaces which satisfies the strong energy condition. Then there exists a Cauchy temporal function $\tau \in C^\infty(M)$ such that one of the families $\{\ell(\cdot,\Sigma_t)\}_{t > 0}$ or $\{\ell(\Sigma_{-t},\cdot)\}_{t > 0}$ is locally equi-Lipschitz, where $\Sigma_t := \{\tau =t\}$.
\end{Conjecture}

The counterexamples we gave to the equi-Lipschitzness of these families were (1) in a setting where the strong energy condition holds but the Cauchy surfaces are not compact (Minkowski), and (2) in a setting where Cauchy surfaces are compact but the strong energy condition does not hold (the example by Ehrlich--Galloway). In light of this, we formulate the following conjecture, which would imply our previous one:

\begin{Conjecture}
    Let $(M,g)$ be a globally hyperbolic, timelike geodesically complete spacetime with compact Cauchy surfaces which satisfies the strong energy condition. Then for every Cauchy temporal function $\tau \in C^\infty(M)$ both of the families $\{\ell(\cdot,\Sigma_t)\}_{t > 0}$ or $\{\ell(\Sigma_{-t},\cdot)\}_{t > 0}$ are locally equi-Lipschitz.
\end{Conjecture}

To see that this is not unreasonable, let us prove the following:

\begin{Proposition}\label{P:compact splitting implies equi-Lipschitz}
    Let $(M,g) = (\R \times S, dt^2 - \tilde g)$ be a product spacetime, where $(S,\tilde g)$ is a compact Riemannian manifold. Then for every Cauchy temporal function $\tau \in C^\infty(\R \times S)$, the families $\{\ell(\cdot,\Sigma_t)\}_{t > 0}$ and $\{\ell(\Sigma_{-t},\cdot)\}_{t > 0}$ are locally equi-Lipschitz.
\end{Proposition}
\begin{proof}
    We show this for $\{\ell(\cdot,\Sigma_t)\}_{t > 0}$, the case of the other family is analogous.

    Given any $s \in \R$, since $\Sigma_s$ is a smooth spacelike acausal Cauchy hypersurface it is easily seen that it must be a global graph over $S$, i.e.,
    \begin{equation}
        \Sigma_s = \{(f_s(x),x) \in \R \times S : x \in S\}
    \end{equation}
    for some smooth function $f_s \in C^\infty(S)$. The fact that $\Sigma_s$ is spacelike translates to the condition $|df_s|_{\tilde g} < 1$, i.e., $f_s$ is $1$-Lipschitz on $S$. Let $D:=\mathrm{diam}(S) < +\infty$ (by compactness), then
    \begin{equation}
        \max_S f_s - \min_S f_s \leq D <+\infty
    \end{equation}
    for every $s \in \R$. We write $M_s:=\max_S f_s$, so that $f_s \geq M_s - D$. Fix a compact set $K \subseteq \R \times S$ and take any $p:=(t_0,x_0) \in K$. For $s$ large enough, $K \subseteq I^-(\Sigma_s)$. Connect $p$ to $\Sigma_s$ by the vertical segment $\alpha:t \mapsto (t + t_0,x_0)$. This segment uniquely meets $\Sigma_s$ at $t = f_s(x_0)$, so that its Lorentzian arclength 
    $ L_g(\alpha)$ is
    \begin{equation}
       \ell(p,\Sigma_s) \geq L_g(\alpha) = f_s(x_0) - t_0 \geq M_s - D - t_0.
    \end{equation}
    Let $\gamma_s(t) = (t,c_s(t))$ be any timelike geodesic realizing $\ell(p,\Sigma_s)$, so that $v_s:=| c'_s|_{\tilde g} < 1$. This geodesic intersects $\Sigma_s$ at some $(f_s(y_s),y_s)$, so that
    \begin{equation}
        \ell(p,\Sigma_s) = L_g(\gamma_s) = (f_s(y_s) - t_0) \sqrt{1 - v_s^2} \leq (M_s - t_0) \sqrt{1 - v_s^2}.
    \end{equation}
    In total,
    \begin{equation}
        M_s - D - t_0 \leq (M_s - t_0) \sqrt{1 - v_s^2}.
    \end{equation}
    Note that $M_s \to +\infty$ as $s \to +\infty$, and moreover $t_0$ remains uniformly bounded for $p = (t_0,x_0) \in K$ since the latter is compact. Thus, dividing by $M_s$ and sending $s \to \infty$, we get $1 \leq \liminf_{s \to \infty} \sqrt{1 - v_s^2}$, so that $\lim_{s \to \infty} v_s^2 \to 0$. This shows that $|\gamma_s'(0)|_h^2 = 1 + v_s^2$ remains uniformly bounded in $s$, where $h:=dt^2 + \tilde g$, thus we conclude the proof upon recalling Lemma \ref{Lemma: blowup of initial tangents}.
\end{proof}

The conjectures formulated above provide a different perspective on Bartnik's conjecture, which could be helpful in its study. In light of the previous result, they are equivalent to each other and to Bartnik's conjecture itself.

\section*{Data availability statement}

There is no data associated with this manuscript.

\section*{Conflict of interest statement}

The authors confirm that no conflict of interest exists for this work.

\section*{Acknowledgments}

Greg Galloway would like to thank Eric Ling and Carlos Vega for helpful discussions.

Robert McCann's work was supported in part by the Canada Research Chairs program CRC-2020-00289 and Natural Sciences and Engineering Research Council of Canada Discovery Grants RGPIN--2020--04162 and 2026-04906. 

This research was funded in part by the Austrian Science Fund (FWF) [Grants DOI \href{https://doi.org/10.55776/EFP6}{10.55776/EFP6} and \href{https://doi.org/10.55776/J4913}{10.55776/J4913}]. For open access purposes, the authors have applied a CC BY public copyright license to any author accepted manuscript version arising from this submission.

\addcontentsline{toc}{section}{References}
\bibliography{Bibliography} 
\bibliographystyle{abbrv}

\end{document}